\documentclass{amsart}
\usepackage{amssymb}
\usepackage{amsthm}
\usepackage{url}
\usepackage{mathtools}
\mathtoolsset{showonlyrefs=true}
\allowdisplaybreaks
\usepackage{Article}
\usepackage{Symbols}
\usepackage{hyperref}
\usepackage[
	backend=biber,
	style=alphabetic,
	maxalphanames=4,
	sorting=nty,
	arxiv=abs,
	doi=true,
	url=true,
	isbn=false,
	eprint=true
]{biblatex}
\usepackage{xurl}
\usepackage{needspace}

\makeatletter
\renewcommand{\subsection}{ \@startsection{subsection}{2}{\z@}  {.5\linespacing\@plus.7\linespacing}  {.5\linespacing}  {\normalfont\bfseries}}
\makeatother

\counterwithin{equation}{section}

\renewcommand{\thetheorem}{\thesection.\arabic{theorem}}

\makeatletter
\newcommand{\undeclare@theorem}[1]{%
	\expandafter\let\csname #1\endcsname\relax
	\expandafter\let\csname end#1\endcsname\relax
}
\@for\@tempa:={theorem,definition,lemma,proposition,corollary,conj,remark,note,eg,fact}\do{%
	\undeclare@theorem{\@tempa}%
	\undeclare@theorem{\@tempa*}%
}
\let\c@theorem\relax
\let\thetheorem\relax
\makeatother

\theoremstyle{plain}
\newtheorem{theorem}{Theorem}[section]
\newtheorem{lemma}[theorem]{Lemma}
\newtheorem{proposition}[theorem]{Proposition}
\newtheorem{corollary}[theorem]{Corollary}
\newtheorem{fact}[theorem]{Fact}

\newtheorem*{theorem*}{Theorem}
\newtheorem*{lemma*}{Lemma}
\newtheorem*{proposition*}{Proposition}
\newtheorem*{corollary*}{Corollary}
\newtheorem*{conj*}{Conjecture}

\theoremstyle{definition}

\newtheorem{remark}[theorem]{Remark}

\newtheorem*{definition*}{Definition}

\DeclareFieldFormat[article]{book}{\mkbibemph{#1}}
\DeclareFieldFormat*{volume}{\mkbibbold{#1}}
\renewbibmacro{in:}{}
\DeclareFieldFormat{pages}{#1}

\renewbibmacro*{journal+issuetitle}{%
	\usebibmacro{journal}%
	\setunit*{\addspace}%
	\iffieldundef{series}
		{}
		{\newunit\printfield{series}\setunit{\addspace}}%
	\usebibmacro{volume+number+eid}%
	\setunit{\addspace}%
	\usebibmacro{issue+date}%
	\setunit{\addcolon\space}%
	\usebibmacro{issue}%
	\newunit
}

\DeclareFieldFormat*{title}{\mkbibemph{#1}}
\DeclareFieldFormat*{journaltitle}{\textnormal{#1}}

\AtEveryBibitem{%
	\clearfield{note}%
	\iffieldundef{eprint}{}{%
		\clearfield{howpublished}\clearfield{url}%
		\ifentrytype{misc}{\clearfield{doi}}{}%
	}%
	\iffieldundef{doi}{}{\iffieldundef{addendum}{\clearfield{url}}{}}%
	\iffieldundef{url}{\clearfield{urlyear}\clearfield{urlmonth}\clearfield{urlday}}{}%
}

\newtheorem{example}[theorem]{Example}

\title[Finite Hypergeometric Reductions of Lauricella's \(F_C\)]{A Criterion for Finite Hypergeometric Reductions of Lauricella's \(F_C\) with Arguments Related by Cube Roots of Unity}
\author{{\sc Nakano} Ryunosuke}
\address{Graduate School of Science, Hokkaido University, Sapporo 060-0810, Japan}
\email{nakano.ryunosuke.i3@elms.hokudai.ac.jp}

\begin{document}
\hypersetup{hidelinks}
\begin{abstract}
We consider the germ at the origin of Lauricella's function \(F_C^{(3)}(a,b;c_0,c_1,c_2;\kappa,\omega\kappa,\omega^2\kappa)\), where \(\omega=\exp(2\pi\sqrt{-1}/3)\).
We assume that \(a,b,c_0,c_1,c_2\) avoid nonpositive integers and that \(c_0,c_1,c_2\) also avoid half-integers.
We consider finite linear combinations of ordinary generalized hypergeometric series with arguments \(\alpha\kappa^3\), where \(\alpha\) is a nonzero complex constant that may vary between summands, each multiplied by a nonnegative integer power of \(\kappa\).
Every target series has lower parameters outside the nonpositive integers and converges near the origin or terminates.
We prove that such a representation exists if and only if, for each \(j=0,1,2\), either \(c_j-c_0\) or \(c_j+c_0\) is an integer, and construct finite sums of ordinary \({}_9F_8\) functions when this condition holds.
We prove necessity from irreducible differential equations of order \(24\) for the generating functions of the three coefficient subsequences and from their formal solutions at infinity.
\end{abstract}
\maketitle
\raggedbottom

\section{Introduction}
\label{sec:introduction}
For \(u\in\C\) and \(n\in\Z_{\geq0}\), we put \((u)_0=1\) and \((u)_n=u(u+1)\cdots(u+n-1)\).
Lauricella's function \(F_C\) in three variables is defined by
\begin{equation}\label{eq:fc-definition}
 F_C^{(3)}(a,b;c_0,c_1,c_2;x_0,x_1,x_2)
 =\sum_{n_0,n_1,n_2\geq0}
 \frac{(a)_{n_0+n_1+n_2}(b)_{n_0+n_1+n_2}}
 {\prod_{j=0}^2(c_j)_{n_j}n_j!}
 \prod_{j=0}^2x_j^{n_j}.
\end{equation}
Here \(a,b,c_0,c_1,c_2,x_0,x_1,x_2\in\C\), and \(c_j\notin\Z_{\leq0}\) for \(j=0,1,2\).
The series converges absolutely when \(\sum_{j=0}^2\sqrt{|x_j|}<1\).
We put \(\omega=\exp(2\pi\sqrt{-1}/3)\) and consider the germ
\begin{equation}\label{eq:restriction}
 F_C^{(3)}(a,b;c_0,c_1,c_2;\kappa,\omega\kappa,\omega^2\kappa).
\end{equation}
Here and throughout the paper, a germ is taken at the origin.

We use the ordinary generalized hypergeometric series
\begin{equation}\label{eq:pfq-definition}
 {}_pF_q\left(\begin{matrix}\boldsymbol u\\ \boldsymbol v\end{matrix};x\right)
 =\sum_{n\geq0}
 \frac{\prod_{i=1}^p(u_i)_n}{\prod_{j=1}^q(v_j)_n}\frac{x^n}{n!}.
\end{equation}
Here \(p,q\in\Z_{\geq0}\), \(x\in\C\), \(\boldsymbol u=(u_1,\ldots,u_p)\in\C^p\), and \(\boldsymbol v=(v_1,\ldots,v_q)\in\C^q\).
Every lower parameter \(v_j\) avoids \(\Z_{\leq0}\).
We only use series that converge in a neighborhood of \(x=0\), or series that terminate.
Thus a nonterminating series satisfies \(p\leq q+1\).
In particular, we do not define the target functions by limits at singular lower parameters.

Henrici's triple product formula gives a single \({}_2F_7\) function when \(c_0=c_1=c_2\); see \cite{Henrici1987} and \cite[Eq.~(1.7)]{ChoiWangRathie2013Triple}.
Reductions obtained by multiplying its coefficients by factors depending on the total degree are given in \cite[Eq.~(2.1)]{ChoiWangRathie2013Triple}.
These formulas provide finite reductions for specified parameters.
However, they do not give a necessary condition for a finite sum of arbitrary ordinary generalized hypergeometric functions in \(\kappa^3\).
In this paper, we give a necessary and sufficient condition for such a finite sum.

We state the condition with all restrictions on the target functions.
\begin{theorem}\label{thm:main}
Suppose that
\begin{equation}\label{eq:main-assumptions}
 a,b\notin\Z_{\leq0},
 \qquad
 c_j\notin\Z_{\leq0}\cup\left(\tfrac12+\Z\right)
 \quad(j=0,1,2).
\end{equation}
The following conditions are equivalent.
\begin{enumerate}
\item There is an identity of germs
\begin{equation}\label{eq:finite-sum}
 F_C^{(3)}(a,b;c_0,c_1,c_2;\kappa,\omega\kappa,\omega^2\kappa)
 =\sum_{\ell=1}^L A_\ell\kappa^{e_\ell}
 {}_{p_\ell}F_{q_\ell}
 \left(\begin{matrix}\boldsymbol u_\ell\\ \boldsymbol v_\ell\end{matrix};
                 \alpha_\ell\kappa^3\right).
\end{equation}
Here \(L\in\Z_{>0}\), \(p_\ell,q_\ell\in\Z_{\geq0}\), \(A_\ell\in\C\), \(e_\ell\in\Z_{\geq0}\), and \(\alpha_\ell\in\C^\times\).
The lists \(\boldsymbol u_\ell=(u_{\ell,1},\ldots,u_{\ell,p_\ell})\) and \(\boldsymbol v_\ell=(v_{\ell,1},\ldots,v_{\ell,q_\ell})\) consist of complex constants independent of \(\kappa\).
Every \(v_{\ell,j}\) avoids \(\Z_{\leq0}\), and every target series converges near the origin or terminates.
\item For each \(j=0,1,2\), we have
\begin{equation}\label{eq:integer-condition}
 c_j-c_0\in\Z
 \quad\hbox{or}\quad
 c_j+c_0\in\Z.
\end{equation}
\end{enumerate}
When \eqref{eq:integer-condition} holds, the construction gives a finite sum of ordinary \({}_9F_8\) functions with argument \(64\kappa^3\) and nonnegative integer powers of \(\kappa\).
Some summands reduce to \({}_8F_7\) functions by cancellation of equal upper and lower parameters.
The finite sum is holomorphic on \(|\kappa|<1/4\) and agrees there with a continuation of the germ \eqref{eq:restriction}.
\end{theorem}

\begin{remark}\label{rem:scope}
The target class in \eqref{eq:finite-sum} allows different orders and different nonzero constants \(\alpha_\ell\) in different summands.
Polynomial factors give the same class after expansion into monomials.
Theorem~\ref{thm:main} concerns arguments of the form \(\alpha_\ell\kappa^3\) and target functions whose lower parameters avoid \(\Z_{\leq0}\).
It makes no assertion about general algebraic arguments, rational factors on the \(F_C\) side, parameter derivatives, or limits of sums with singular target parameters.
\end{remark}

Condition \eqref{eq:integer-condition} allows reflection of a lower parameter, in addition to an integer shift.
We can write that condition as
\begin{equation}\label{eq:signed-shifts}
 c_j=1+\varepsilon_j(c_0-1)+m_j,
 \qquad \varepsilon_j\in\{1,-1\},\quad m_j\in\Z.
\end{equation}
We prove sufficiency by contiguous relations for \({}_0F_1\), the scalar differential equation satisfied by the product in Henrici's formula, and removal of finitely many negative powers of an auxiliary variable.
Throughout the domain \eqref{eq:main-assumptions}, this construction gives target functions whose lower parameters avoid \(\Z_{\leq0}\); see Proposition~\ref{prop:construction}.

For necessity, we remove the two Pochhammer factors involving the upper parameters from the coefficients of \eqref{eq:restriction}.
With an auxiliary variable \(z\), we obtain the product \(\prod_{j=0}^2{}_0F_1(;c_j;\omega^jz)\).
We separate the coefficient sequence into the three residue classes modulo \(3\) and prove that each resulting generating function satisfies an irreducible differential equation of order \(24\) when \eqref{eq:integer-condition} fails; see Theorem~\ref{thm:rank24}.
After the substitution \(w=z^3\), the formal decomposition at infinity has six exponential factors with nonzero exponents, each accompanied by a regular singular factor of rank \(3\).
For an ordinary generalized hypergeometric equation with the same positive slope, each exponential factor with nonzero exponent is accompanied by a regular singular factor of rank \(1\).
Irreducibility then excludes finite sums as well as single generalized hypergeometric functions; see Proposition~\ref{prop:obstruction}.

\begin{remark}\label{rem:bessel-background}
We use the differential Galois correspondence and the relation between differential modules and representations in \cite[Theorem~1.3.9, Proposition~1.3.28, and Theorem~1.5.2]{Singer2009}.
For formal solutions we use \cite[\S1.4]{Singer2009}.
The numbered locators refer to the author manuscript arXiv:0712.4124v2.
We compute the differential Galois group of Bessel's equation and prove the isomorphism criterion needed here, both over the rational function field and after algebraic extension.
\end{remark}

\section{Coefficient operations and differential equations}
\label{sec:preliminaries}
In this section, we give the coefficient operations and the facts about differential equations used in the proof.

\subsection{Grouping coefficients}
We define the coefficients \(q_N\) by
\begin{equation}\label{eq:product-coefficients}
 \prod_{j=0}^2{}_0F_1(;c_j;\omega^jz)=\sum_{N\geq0}q_Nz^N.
\end{equation}
Since every factor is entire, the product is entire.
By grouping the absolutely convergent series \eqref{eq:fc-definition} by total degree, we obtain
\begin{equation}\label{eq:grouping}
 F_C^{(3)}(a,b;c_0,c_1,c_2;\kappa,\omega\kappa,\omega^2\kappa)
 =\sum_{N\geq0}(a)_N(b)_Nq_N\kappa^N
 \qquad(|\kappa|<1/9).
\end{equation}
We put
\begin{equation}\label{eq:delta}
 \Delta_h(u)=
 \left(\frac{u+h}{3},\frac{u+h+1}{3},\frac{u+h+2}{3}\right)
 \qquad(h\in\Z_{\geq0}).
\end{equation}
For a parameter list \(\boldsymbol u\), we use \((\boldsymbol u)_n=\prod_i(u_i)_n\) and \(\boldsymbol u+s=(u_i+s)_i\).
We use concatenation for parameter lists in a generalized hypergeometric series.

The following coefficient identities give both directions of the proof.
\begin{lemma}\label{lem:coefficient-operations}
For a formal power series \(\sum_{N\geq0}f_Nz^N\), multiplication of its coefficient of degree \(N\) by \((a)_N(b)_N\) sends
\begin{equation}\label{eq:forward-operation}
 z^h{}_pF_q
 \left(\begin{matrix}\boldsymbol u\\ \boldsymbol v\end{matrix};\beta z^3\right)
 \longmapsto
 (a)_h(b)_h\kappa^h{}_{p+6}F_q
 \left(\begin{matrix}\boldsymbol u,\Delta_h(a),\Delta_h(b)\\
                      \boldsymbol v\end{matrix};729\beta\kappa^3\right),
\end{equation}
where \(h\in\Z_{\geq0}\).
If \(a,b\notin\Z_{\leq0}\), division of the coefficient of degree \(N\) by \((a)_N(b)_N\) sends
\begin{equation}\label{eq:inverse-operation}
 \kappa^h{}_pF_q
 \left(\begin{matrix}\boldsymbol u\\ \boldsymbol v\end{matrix};\alpha\kappa^3\right)
 \longmapsto
 \frac{z^h}{(a)_h(b)_h}{}_pF_{q+6}
 \left(\begin{matrix}\boldsymbol u\\
             \boldsymbol v,\Delta_h(a),\Delta_h(b)\end{matrix};
             \frac{\alpha z^3}{729}\right).
\end{equation}
The additional lower parameters in \eqref{eq:inverse-operation} avoid \(\Z_{\leq0}\).
\end{lemma}
\begin{proof}
We use
\[
 (a)_{h+3n}=(a)_h\,3^{3n}
 \prod_{\delta\in\Delta_h(a)}(\delta)_n
\]
and the same equality for \(b\).
These are the Pochhammer forms of the multiplication formula \cite[Eq.~5.5.6]{DLMF}; see also \cite[\S5.5]{NIST2010}.
Comparison of the coefficient of degree \(h+3n\) gives both identities.
If an entry of \(\Delta_h(a)\) is a nonpositive integer, then \(a=-h-j-3n\) for some \(j\in\{0,1,2\}\) and \(n\geq0\).
This contradicts \(a\notin\Z_{\leq0}\).
The argument for \(b\) is the same.
\end{proof}

\begin{remark}\label{rem:coefficient-convergence}
The identities in Lemma~\ref{lem:coefficient-operations} are identities of formal series before any assertion of convergence.
In \eqref{eq:inverse-operation}, a convergent nonterminating target has \(p\leq q+1\), so the transformed series is entire.
A terminating series remains a polynomial.
In the construction used to prove sufficiency, \eqref{eq:forward-operation} sends a \({}_3F_8\) series with argument \(64z^3/729\) to a \({}_9F_8\) series with argument \(64\kappa^3\).
The resulting series converges when \(|\kappa|<1/4\).
\end{remark}

\subsection{Differential modules}
For a variable \(x\), we write \(\partial_x=d/dx\).
A differential module over \(\C(x)\) is a vector space of finite dimension over \(\C(x)\) with a derivation satisfying the product rule.
For a function \(f\) satisfying a scalar linear differential equation, we put
\[
 \mathcal N_f=\C(x)f+\C(x)\partial_xf+\C(x)\partial_x^2f+\cdots.
\]
The rank of \(\mathcal N_f\) is the order of the monic minimal scalar equation for \(f\).
We call a differential module irreducible if its only differential submodules are zero and the whole module.
For a scalar equation we use irreducibility in the corresponding sense of factorization over \(\C(x)[\partial_x]\).
We state formal expansions as expansions of solutions of scalar equations.
The representation on these scalar solutions is dual to the representation on horizontal solutions of the cyclic module \(\mathcal N_f\).
Duality negates the exponents and preserves their associated ranks, so the comparisons of subquotients below apply to either convention.

We record two consequences of the representation theory of differential modules.
\begin{lemma}\label{lem:module-operations}
The following statements hold over \(\C(x)\).
\begin{enumerate}
\item Suppose that \(\mathcal N_1,\ldots,\mathcal N_s\) are pairwise nonisomorphic irreducible differential modules.
A vector in \(\bigoplus_i\mathcal N_i\) generates the whole direct sum if its projection to every summand is nonzero.
\item Suppose that \(f=\sum_{\ell=1}^L R_\ell(x)f_\ell\), where \(R_\ell\in\C(x)\) and each \(f_\ell\) satisfies a scalar equation.
If \(\mathcal N_f\) is irreducible, then \(\mathcal N_f\) is a composition factor of the differential module associated with an equation for at least one \(f_\ell\).
\end{enumerate}
\end{lemma}
\begin{proof}
For the first assertion, a submodule of a finite direct sum of simple, pairwise nonisomorphic modules is a sum of some of the summands.
The projection condition includes every summand.
For the second assertion, we discard terms with \(R_\ell=0\).
Multiplication by a nonzero rational function preserves the differential module generated by a function, since \(\mathcal N_{R_\ell f_\ell}=\mathcal N_{f_\ell}\) as subspaces of the ambient differential field.
The addition map from \(\bigoplus_\ell\mathcal N_{R_\ell f_\ell}\) has an image containing \(\mathcal N_f\).
Thus \(\mathcal N_f\) is a subquotient of a finite direct sum of the modules for the equations of the \(f_\ell\).
By a composition series, a simple subquotient occurs in one of the summands.
\end{proof}

We also use the formal decomposition of a differential equation after a finite ramification at a singular point; see \cite[\S1.4, the formal fundamental matrix preceding Example~1.4.3]{Singer2009}.
At \(x=\infty\), after writing \(x=s^d\), the formal differential module is a direct sum of modules
\[
 \mathcal E_\varphi\otimes\mathcal R_\varphi.
\]
Here \(\varphi\) is a polynomial without constant term, \(\mathcal E_\varphi\) is a differential module of rank \(1\) whose horizontal equation is \(dy/ds=\varphi'(s)y\), and \(\mathcal R_\varphi\) is regular singular.
The equation for \(\mathcal E_\varphi\) has the solution \(e^{\varphi(s)}\).
The rank attached to \(\varphi\) is the rank of \(\mathcal R_\varphi\).
For \(\varphi\ne0\), the slope is \(\deg(\varphi)/d\).
The factor with \(\varphi=0\) is regular singular.
The decomposition is compatible with subquotients.
Consequently, we cannot increase the rank attached to a fixed exponential factor by passing to a subquotient.
These exponential factors and their associated ranks do not change when we tensor with a differential module of rank \(1\) whose equation has a solution \(x^\rho\), or when we multiply a function by a rational function.

\section{Construction of the finite sums}
\label{sec:construction}
In this section, we prove the sufficient direction of Theorem~\ref{thm:main} and ensure that the lower parameters of the resulting functions avoid nonpositive integers.

\subsection{The scalar equation for the product}
We put \(g=c-1\) and \(\theta=z\,d/dz\).
For \(j=0,1,2\), we take any solution \(y_j\) of
\begin{equation}\label{eq:common-bessel}
 \theta(\theta+g)y_j=\omega^jzy_j.
\end{equation}
For \(S\subset\{0,1,2\}\), we put
\[
 E_S=\prod_{j\in S}\theta y_j\prod_{j\notin S}y_j.
\]
We define
\[
 \mathcal H=E_\varnothing,\qquad
 \mathcal D=E_{\{0,1,2\}},\qquad
 \mathcal U_r=\sum_{j=0}^2\omega^{rj}E_{\{j\}},\qquad
 \mathcal V_r=\sum_{j=0}^2\omega^{rj}E_{\{0,1,2\}\setminus\{j\}},
\]
where \(r=0,1,2\).

The product rule gives the following equation for every choice of the three solutions.
\begin{lemma}\label{lem:product-equation}
We put
\begin{equation}\label{eq:indicial-polynomial}
 P_c(T)=T(T+3c-3)
       \prod_{j=1}^3(T+c-j)
       \prod_{j=2}^4(T+2c-j).
\end{equation}
Then
\begin{equation}\label{eq:scalar-product}
 \left[
 P_c(\theta)-64z^3
 \left(\theta+\frac{3(2c-1)}4\right)
 \left(\theta+\frac{3(2c+1)}4\right)
 \right]\mathcal H=0.
\end{equation}
\end{lemma}
\begin{proof}
We differentiate the eight products to obtain
\[
 \theta E_S
 =\sum_{j\notin S}E_{S\cup\{j\}}
  +z\sum_{j\in S}\omega^jE_{S\setminus\{j\}}
  -|S|gE_S.
\]
Substituting the definitions of the displayed sums gives
\begin{equation}\label{eq:eight-system}
\begin{array}{ll}
 \theta\mathcal H=\mathcal U_0,
 &(\theta+3g)\mathcal D=z\mathcal V_1,\\
 (\theta+g)\mathcal U_0=2\mathcal V_0,
 &(\theta+2g)\mathcal V_0=3\mathcal D-z\mathcal U_1,\\
 (\theta+g)\mathcal U_1=-\mathcal V_1,
 &(\theta+2g)\mathcal V_1=2z\mathcal U_2,\\
 (\theta+g)\mathcal U_2=3z\mathcal H-\mathcal V_2,
 &(\theta+2g)\mathcal V_2=-z\mathcal U_0.
\end{array}
\end{equation}
We put
\[
 P_0(T)=T(T+g)(T+2g)(T+3g),\qquad
 T_0(T)=(T+g-1)(T+g)(T+2g-1)(T+2g).
\]
Eliminating \(\mathcal U_0,\mathcal V_0,\mathcal D\) gives
\[
 P_0(\theta)\mathcal H
 =-2z(4\theta+1+6g)\mathcal U_1.
\]
Here we use
\[
 \mathcal V_0=\tfrac12(\theta+g)\theta\mathcal H,\qquad
 3\mathcal D=(\theta+2g)\mathcal V_0+z\mathcal U_1,\qquad
 (\theta+3g)\mathcal D=-z(\theta+g)\mathcal U_1.
\]
Eliminating \(\mathcal V_1,\mathcal U_2,\mathcal V_2\) gives
\[
 T_0(\theta)\mathcal U_1=-2z^2(4\theta+3+6g)\mathcal H.
\]
For this elimination we use
\[
 2z\mathcal U_2=-(\theta+2g)(\theta+g)\mathcal U_1,\qquad
 (\theta+2g)(\theta+g)\mathcal U_2=z(4\theta+3+6g)\mathcal H.
\]
Apply \(T_0(\theta-1)\) to the first eliminated equation and use \(\theta z=z(\theta+1)\).
We obtain \eqref{eq:scalar-product}, since \(P_0(T)T_0(T-1)=P_c(T)\).
All these operations are polynomial identities in \(c,z,\theta\).
\end{proof}

\subsection{Two initial products}
We put
\begin{align*}
 \boldsymbol B(c)&=
 \left(c,\frac c3,\frac{c+1}3,\frac{c+2}3,
       \frac{2c-1}3,\frac{2c}3,\frac{2c+1}3\right),\\
 \boldsymbol A(c)&=\left(\frac{2c-1}4,\frac{2c+1}4\right),
 \qquad K=\frac{64}{729}.
\end{align*}
We use Henrici's formula under the following restrictions on its parameter.
\begin{fact}\label{fact:henrici}
Suppose that \(c\in\Z_{>0}\), or that \(c\notin\Z\cup(\tfrac12+\Z)\).
Then
\begin{equation}\label{eq:henrici}
 \prod_{j=0}^2{}_0F_1(;c;\omega^jz)
 ={}_2F_7\left(\begin{matrix}\boldsymbol A(c)\\
                         \boldsymbol B(c)\end{matrix};Kz^3\right).
\end{equation}
\end{fact}
\begin{proof}
We use Henrici's triple product theorem \cite{Henrici1987}, in the form \cite[Eq.~(1.7)]{ChoiWangRathie2013Triple} with its parameter \(6c\) replaced by \(c\).
Under the stated restrictions, every lower parameter in \(\boldsymbol B(c)\) avoids nonpositive integers.
We can also verify \eqref{eq:henrici} by comparing coefficients in \eqref{eq:scalar-product} and using the invariance under \(z\mapsto\omega z\) and the constant coefficient \(1\).
\end{proof}

For the second product, we put
\begin{align*}
 \gamma_0(c)&=1,\qquad
 \gamma_1(c)=\frac{2(c-1)}{c(2-c)},\\
 \gamma_2(c)&=\frac{(c-1)(2c+3)}{c^2(c-3)(c+1)},\\
 \boldsymbol A_r(c)&=
 \left(\frac{2c+1+4r}{12},\frac{2c+7+4r}{12}\right)
 \qquad(r=0,1,2),
\end{align*}
and
\begin{align*}
 \boldsymbol B_0(c)&=
 \left(\frac13,\frac23,\frac c3,\frac{c+1}3,\frac{c+2}3,
                  \frac{4-c}3,\frac{2c+1}3\right),\\
 \boldsymbol B_1(c)&=
 \left(\frac23,\frac43,\frac{c+1}3,\frac{c+2}3,\frac{c+3}3,
                  \frac{5-c}3,\frac{2c+2}3\right),\\
 \boldsymbol B_2(c)&=
 \left(\frac43,\frac53,\frac{c+2}3,\frac{c+3}3,\frac{c+4}3,
                  \frac{6-c}3,\frac{2c+3}3\right).
\end{align*}
The scalar equation gives the following identity with lower parameter \(2-c\).
\begin{proposition}\label{prop:reflected-product}
Suppose that \(c\notin\Z\cup(\tfrac12+\Z)\).
Then
\begin{equation}\label{eq:reflected-product}
 {}_0F_1(;2-c;z){}_0F_1(;c;\omega z){}_0F_1(;c;\omega^2z)
 =\sum_{r=0}^2\gamma_r(c)z^r{}_2F_7
 \left(\begin{matrix}\boldsymbol A_r(c)\\
                     \boldsymbol B_r(c)\end{matrix};Kz^3\right).
\end{equation}
Every lower parameter on the right avoids \(\Z_{\leq0}\).
Moving \(2-c\) to position \(j\), with positions numbered \(0,1,2\), multiplies the summand of residue \(r\) by \(\omega^{jr}\).
\end{proposition}
\begin{proof}
We take
\[
 y_0=z^{1-c}{}_0F_1(;2-c;z),\qquad
 y_1={}_0F_1(;c;\omega z),\qquad
 y_2={}_0F_1(;c;\omega^2z).
\]
The three functions satisfy \eqref{eq:common-bessel}.
Thus \(z^{1-c}\) times the product on the left of \eqref{eq:reflected-product} satisfies \eqref{eq:scalar-product}.
The roots of \(P_c\) are
\[
 0,\quad 3-3c,\quad 1-c,2-c,3-c,\quad
 2-2c,3-2c,4-2c.
\]
Substitute \(z^{1-c+r}\sum_{n\geq0}v_nz^{3n}\) in that equation.
For \(n\geq1\), comparison of coefficients gives
\[
 \frac{v_n}{v_{n-1}}
 =K\,\frac{(\boldsymbol A_r(c)+n-1)_1}
              {n(\boldsymbol B_r(c)+n-1)_1}.
\]
None of the listed lower parameters is a nonpositive integer: such an entry would force \(c\) to be an integer or a half-integer.
The recurrence therefore determines the series from \(v_0\).
Direct multiplication of the first three terms of the three \({}_0F_1\) series gives \(\gamma_0(c),\gamma_1(c),\gamma_2(c)\).
Both sides are entire in \(z\).
Cyclic permutation of the factors is the substitution \(z\mapsto\omega^jz\), which gives the last assertion.
\end{proof}

\subsection{Independent integer shifts}
We use the following recovery of all eight components from the product.
\begin{lemma}\label{lem:recovery}
Write \eqref{eq:eight-system} as \(\theta Y=M(z)Y\), in the order
\[
 Y={}^t(\mathcal H,\mathcal U_0,\mathcal V_0,\mathcal D,
                \mathcal U_1,\mathcal V_1,\mathcal U_2,\mathcal V_2).
\]
We put \(v_0=(1,0,\ldots,0)\), \(v_{k+1}=zv_k'+v_kM\), and let \(\mathcal C\) have rows \(v_0,\ldots,v_7\).
Then
\begin{equation}\label{eq:recovery-determinant}
 \det\mathcal C
 =768z^6(2g-5)(2g-1)(2g+1)(2g+5).
\end{equation}
If this determinant is not identically zero, each \(E_S\) is a linear combination of \(\theta^k\mathcal H\), \(0\leq k\leq7\), with coefficients that are Laurent polynomials and depend rationally on \(c,\omega\).
\end{lemma}
\begin{proof}
The recurrence defining \(v_k\) gives \(\theta^k\mathcal H=v_kY\).
We evaluate the determinant using the eight entries of \eqref{eq:eight-system}, which gives \eqref{eq:recovery-determinant}.
The entries of \(\mathcal C\) are polynomials in \(z,g\).
Its inverse therefore has entries in \(\Q(g)[z,z^{-1}]\) whenever the four parameter factors are nonzero.
We recover the eight \(E_S\) by inverting the finite sums defining the components of \(Y\).
\end{proof}

We now give all steps of the finite construction.
\begin{proposition}\label{prop:construction}
Under \eqref{eq:main-assumptions} and \eqref{eq:integer-condition}, the identity \eqref{eq:finite-sum} exists.
We can take every nonzero summand to be an ordinary \({}_9F_8\) function with argument \(64\kappa^3\), multiplied by a nonnegative integer power of \(\kappa\).
The construction uses finitely many rational operations and integer shifts of parameters.
\end{proposition}
\begin{proof}
We consider two cases for \(c_0\).

First suppose that \(c_0\notin\Z\).
We put \(c=c_0\) and write the three lower parameters as in \eqref{eq:signed-shifts}.
For \(\delta\in\{c,2-c\}\), contiguous relations give
\begin{equation}\label{eq:contiguous-operators}
 {}_0F_1(;\delta+m;x)
 =A_m(x){}_0F_1(;\delta;x)
    +B_m(x)\theta_x{}_0F_1(;\delta;x),
 \qquad\theta_x=x\frac{d}{dx},
\end{equation}
where \(A_m,B_m\in\Q(\delta)[x,x^{-1}]\).
We start with \(A_0=1,B_0=0\).
For \(\gamma\in\C\) with \(\gamma,\gamma-1\notin\Z_{\leq0}\), the series definition gives
\[
 {}_0F_1(;\gamma+1;x)=\frac{\gamma}{x}\theta_x{}_0F_1(;\gamma;x),
 \qquad
 {}_0F_1(;\gamma-1;x)
 =\left(1+\frac{\theta_x}{\gamma-1}\right){}_0F_1(;\gamma;x).
\]
Using \(\theta_x^2{}_0F_1(;\delta;x)=x{}_0F_1(;\delta;x)+(1-\delta)\theta_x{}_0F_1(;\delta;x)\), we obtain
\begin{align}\label{eq:shift-recurrences}
 A_{m+1}&=\frac{\delta+m}{x}(\theta_xA_m+xB_m),\notag\\
 B_{m+1}&=\frac{\delta+m}{x}
       (A_m+\theta_xB_m+(1-\delta)B_m),\notag\\
 A_{m-1}&=A_m+\frac{\theta_xA_m+xB_m}{\delta+m-1},\notag\\
 B_{m-1}&=B_m+
       \frac{A_m+\theta_xB_m+(1-\delta)B_m}{\delta+m-1}.
\end{align}
Since \(\delta\) is not an integer, all denominators required by these recurrences are nonzero.

We put \(d=\#\{j:\varepsilon_j=-1\}\) and choose
\[
 y_j=
 \begin{cases}
 {}_0F_1(;c;\omega^jz),&\varepsilon_j=1,\\
 z^{-g}{}_0F_1(;2-c;\omega^jz),&\varepsilon_j=-1.
 \end{cases}
\]
These functions satisfy \eqref{eq:common-bessel}.
Their product is \(z^{-dg}\prod_j{}_0F_1(;1+\varepsilon_jg;\omega^jz)\).
When \(\varepsilon_j=-1\), \eqref{eq:contiguous-operators} expresses the shifted function through \(z^gy_j\) and \(z^g\theta y_j\), with \(A_m\) replaced by \(A_m+gB_m\).
Multiplying the three shifted factors gives a linear combination of \(z^{dg}E_S\) with Laurent polynomials as coefficients.
The parameter factors in \eqref{eq:recovery-determinant} are nonzero under our assumption on \(c\).
By Lemma~\ref{lem:recovery}, we can express the shifted product as a finite linear combination of the following functions, with Laurent polynomials as coefficients:
\[
 z^{dg}\theta^k\mathcal H
 =(\theta-dg)^k
       \prod_{j=0}^2{}_0F_1(;1+\varepsilon_jg;\omega^jz).
\]

If \(d=0\) or \(d=3\), we use \eqref{eq:henrici} with parameter \(c\) or \(2-c\).
If \(d=1\), we use \eqref{eq:reflected-product} and its cyclic permutations.
If \(d=2\), replacement of \(c\) by \(2-c\) reduces the initial product to the case \(d=1\).
All lower parameters in these initial products avoid nonpositive integers.
Under \(x=Kz^3\), conjugation by \(z^r\) gives
\[
 z^{-r}\theta z^r=r+3x\frac{d}{dx}
\]
as an identity of operators on functions of \(x\).
Every polynomial in \(x\,d/dx\) is a finite linear combination of the operators \(x^j(d/dx)^j\).
Thus polynomial expressions in \(\theta\) reduce to finitely many ordinary derivatives of the displayed \({}_2F_7\) series.
The differentiation formula \cite[Eq.~16.3.1]{DLMF} shifts every upper and lower parameter by the same nonnegative integer.
Thus we obtain a finite sum of terms
\begin{equation}\label{eq:laurent-summand}
 C_0z^e{}_2F_7
 \left(\begin{matrix}\boldsymbol A\\\boldsymbol B\end{matrix};Kz^3\right),
 \qquad e\in\Z,
\end{equation}
where \(\boldsymbol A\) has two entries, and \(\boldsymbol B\) has seven entries, none of which is a nonpositive integer.
The constant \(C_0\) and both lists come from the finite recurrences and the matrix \(\mathcal C^{-1}\).

For each term \eqref{eq:laurent-summand}, we put
\[
 s=\max(0,\lceil-e/3\rceil),\qquad h=e+3s\geq0,\qquad
 C_s=C_0\frac{K^s(\boldsymbol A)_s}{s!(\boldsymbol B)_s}.
\]
After deleting its terms with negative powers of \(z\), that term becomes
\begin{equation}\label{eq:tail-removal}
 C_sz^h{}_3F_8
 \left(\begin{matrix}1,\boldsymbol A+s\\
                     s+1,\boldsymbol B+s\end{matrix};Kz^3\right).
\end{equation}
The deleted finite Laurent polynomials sum to zero, since the original shifted product is entire.
Every entry of \(\boldsymbol B+s\) avoids nonpositive integers.
Apply \eqref{eq:forward-operation} to \eqref{eq:tail-removal}.
Its contribution is
\begin{equation}\label{eq:constructed-nine}
 C_s(a)_h(b)_h\kappa^h{}_9F_8
 \left(\begin{matrix}
 1,\boldsymbol A+s,\Delta_h(a),\Delta_h(b)\\
 s+1,\boldsymbol B+s
 \end{matrix};64\kappa^3\right).
\end{equation}
This proves the assertion when \(c_0\notin\Z\).

Next suppose that \(c_0\in\Z\).
Condition \eqref{eq:integer-condition} implies that every \(c_j\) is a positive integer, since the lower parameters avoid nonpositive integers.
We choose a positive integer \(C\geq\max_j c_j\).
We start with the product in \eqref{eq:henrici} with all three lower parameters equal to \(C\), and use only downward shifts \(m_j=c_j-C\).
All denominators in the downward recurrences \eqref{eq:shift-recurrences} are positive integers.
The determinant \eqref{eq:recovery-determinant} is nonzero for \(g=C-1\).
The entries of \(\boldsymbol B(C)\) are positive, as are all their nonnegative shifts.
We again recover the components, differentiate the resulting series, and delete the terms with negative powers.
Formula \eqref{eq:constructed-nine} again gives target functions whose lower parameters avoid nonpositive integers.

In both cases, the target \({}_9F_8\) series converge on \(|64\kappa^3|<1\).
Coefficient comparison with \eqref{eq:grouping} proves equality as germs.
The resulting finite sum gives the stated continuation on \(|\kappa|<1/4\).
\end{proof}

\section{Bessel equations and their isomorphisms}
\label{sec:bessel}
In this section, we determine the differential Galois groups and isomorphisms needed for the necessary direction.
All differential fields have constant field \(\C\).

\subsection{The normal form}
For \(\lambda\ne0\), we write \(B_{\lambda,\nu}\) for the differential module associated with
\begin{equation}\label{eq:bessel-normal}
 u''=\left(\lambda+\frac{\nu^2-1/4}{t^2}\right)u,
 \qquad {}'=\frac{d}{dt}.
\end{equation}
If \(\nu=c-1\), direct substitution gives
\begin{equation}\label{eq:normalization}
 u=t^{c-1/2}{}_0F_1\left(;c;\frac{\lambda t^2}{4}\right)
\end{equation}
as a solution of \eqref{eq:bessel-normal}.
The power of \(t\) denotes any local branch on a simply connected domain avoiding \(0\).
The differential module of that power has rank \(1\) over \(\C(t)\).

We give the formal expansions at infinity explicitly.
\begin{lemma}\label{lem:bessel-formal}
Choose \(s\in\C^\times\) with \(s^2=\lambda\).
Equation \eqref{eq:bessel-normal} has formal solutions
\begin{equation}\label{eq:bessel-formal}
 u_\pm=e^{\pm st}\left(1+\sum_{n\geq1}a_n^\pm t^{-n}\right).
\end{equation}
We put \(q=\nu^2-1/4\) and \(a_0^\pm=1\).
The coefficients satisfy
\[
 a_{n+1}^\pm=
 \frac{n(n+1)-q}{\pm2s(n+1)}a_n^\pm
 \qquad(n\geq0).
\]
In particular, the series multiplying the exponentials contain no nonintegral powers of \(t\).
\end{lemma}
\begin{proof}
Substitute \(u=e^{st}F\) in \eqref{eq:bessel-normal}.
The equation becomes \(F''+2sF'=qt^{-2}F\).
Comparison of the coefficient of \(t^{-n-2}\) gives the stated recurrence.
Replacement of \(s\) by \(-s\) gives the other solution.
Their distinct exponentials make them formally independent.
\end{proof}

\subsection{The differential Galois group}
For \(\nu\notin\tfrac12+\Z\), the differential Galois group of Bessel's equation is also recorded in \cite[Example~1.3.32]{Singer2009}.
We prove the version over \(\C(t)\) that we use.
\begin{proposition}\label{prop:bessel-group}
If \(\lambda\ne0\) and \(\nu\notin\tfrac12+\Z\), the differential Galois group of \eqref{eq:bessel-normal} over \(\C(t)\) is \(\SL_2(\C)\).
The differential module remains irreducible after every finite algebraic extension of \(\C(t)\).
\end{proposition}
\begin{proof}
Since the equation has no first derivative term, its Wronskian is a nonzero constant.
The differential Galois group \(G\) is therefore a subgroup of \(\SL_2(\C)\).
Suppose that \(G\) is a proper algebraic subgroup.
Such a subgroup has a finite orbit on \(\Ps^1\).
For a finite subgroup this follows from finiteness.
For a proper subgroup of positive dimension, the identity component is solvable and fixes either one line or a pair of lines; the full subgroup preserves that finite set.
The differential Galois correspondence \cite[Theorem~1.3.9]{Singer2009} then gives a nonzero solution whose logarithmic derivative \(h\) is algebraic over \(\C(t)\).
It satisfies
\begin{equation}\label{eq:riccati}
 h'+h^2=\lambda+\frac{\nu^2-1/4}{t^2}.
\end{equation}

We determine the possible ramification of \(h\).
At infinity, a Puiseux expansion in \eqref{eq:riccati} must begin with \(s\) or \(-s\), where \(s^2=\lambda\).
Coefficient comparison gives
\[
 h=\pm s+\sum_{n\geq2}h_nt^{-n}.
\]
To verify the absence of fractional powers, take the first exponent outside the Laurent series in integral powers of \(t^{-1}\) determined by the chosen sign.
Its coefficient in \(h^2\) is \(\pm2s\) times the proposed coefficient, and no preceding term contributes to that exponent.
Thus the proposed coefficient is zero.
The coefficient of \(t^{-1}\) is zero for the same reason.
This argument applies to every algebraic conjugate of \(h\), so the algebraic function field generated by \(h\) is unramified over infinity.

At a finite point other than \(0\), a finite value of \(h\) gives a holomorphic solution of equation \eqref{eq:riccati}, which has order one.
At a pole we use \(v=1/h\), which satisfies
\[
 v'=1-\left(\lambda+\frac{\nu^2-1/4}{t^2}\right)v^2.
\]
At such a point \(v\) is holomorphic and has derivative \(1\).
Thus every pole away from \(0\) is simple, has residue \(1\), and causes no ramification.
The algebraic function field can have a branch point only at \(0\).
There is no nontrivial connected finite cover of \(\Ps^1\) branched over at most one point.
For example, restriction to the complement of that point gives an unramified finite cover of the simply connected complex plane.
Hence \(h\in\C(t)\).

Since \(\nu\notin\tfrac12+\Z\), we have \(\nu^2-1/4\ne0\).
The equation at \(0\) forces a simple pole with residue \(\rho=1/2+\nu\) or \(\rho=1/2-\nu\).
If \(N\) is the number of the other finite poles, the coefficient of \(t^{-1}\) at infinity gives \(\rho+N=0\).
It follows that \(\nu\in\tfrac12+\Z\), a contradiction.
Thus \(G=\SL_2(\C)\).

A finite algebraic extension does not change the identity component of a differential Galois group; see \cite[Corollary~1.3.10(3)]{Singer2009}.
Since \(\SL_2(\C)\) is connected and its standard representation is irreducible, irreducibility persists after every such extension.
\end{proof}

\subsection{Isomorphisms after algebraic extension}
We need to distinguish the equations even after a finite algebraic extension.
\begin{proposition}\label{prop:bessel-isomorphisms}
Suppose that \(\lambda,\mu\ne0\) and \(\nu,\eta\notin\tfrac12+\Z\).
There is a nonzero differential homomorphism from \(B_{\lambda,\nu}\) to \(B_{\mu,\eta}\) over an algebraic extension of \(\C(t)\) if and only if
\begin{equation}\label{eq:bessel-isomorphism-condition}
 \lambda=\mu,\qquad
 \bigl(\nu-\eta\in\Z\quad\hbox{or}\quad\nu+\eta\in\Z\bigr).
\end{equation}
Every such homomorphism has a matrix over \(\C(t)\).
In particular, the homomorphism is an isomorphism.
\end{proposition}
\begin{proof}
We write both equations as systems of order one for \({}^t(u,u')\).
A homomorphism matrix satisfies their linear Hom equation.
By Proposition~\ref{prop:bessel-group}, a nonzero homomorphism over an algebraic extension is an isomorphism.

By Lemma~\ref{lem:bessel-formal}, we obtain the exponents of the exponential factors in a formal solution of the Hom equation by taking differences of members of
\[
 \{\sqrt{\lambda},-\sqrt{\lambda}\}t
 \quad\hbox{and}\quad
 \{\sqrt{\mu},-\sqrt{\mu}\}t.
\]
An algebraic solution can contain only the factor with zero exponent.
If \(\lambda\ne\mu\), the two sets are disjoint, so there is no nonzero algebraic homomorphism.
If \(\lambda=\mu\), a homomorphism with algebraic entries can have nonzero components only between exponential factors with equal exponents.
We choose matching signs.
The formal fundamental matrices with the exponentials removed have entries in \(\C((t^{-1}))\).
They are invertible over that field.
Thus every entry of an algebraic homomorphism matrix has a Laurent expansion in integral powers of \(t^{-1}\) at infinity.
Every algebraic conjugate satisfies the same Hom equation and has the same property.
At a finite ordinary point the Hom equation has holomorphic coefficients, so all its solutions are single-valued there.
Each entry of the homomorphism matrix can therefore ramify only at \(0\).
The argument about finite covers in Proposition~\ref{prop:bessel-group} shows that every entry is rational.

The local exponents of \eqref{eq:bessel-normal} at \(0\) are \(1/2+\nu,1/2-\nu\).
The eigenvalues of the local monodromy are \(-e^{2\pi\sqrt{-1}\nu}\) and \(-e^{-2\pi\sqrt{-1}\nu}\), including the cases with logarithms.
A rational isomorphism preserves their unordered pair.
Hence \(\nu-\eta\in\Z\) or \(\nu+\eta\in\Z\).

Conversely, the replacement \(\nu\mapsto-\nu\) leaves \eqref{eq:bessel-normal} unchanged.
An integer shift follows from the two operators
\begin{equation}\label{eq:darboux}
 v=u'-\frac{\nu+1/2}{t}u,\qquad
 u=\frac1\lambda\left(v'+\frac{\nu+1/2}{t}v\right).
\end{equation}
The function \(v\) satisfies \eqref{eq:bessel-normal} with \(\nu\) replaced by \(\nu+1\).
The matrix sending \({}^t(u,u')\) to \({}^t(v,v')\) is
\[
 H_\nu=
 \begin{pmatrix}
 -(\nu+1/2)/t&1\\
 \lambda+(\nu+1/2)^2/t^2&-(\nu+1/2)/t
 \end{pmatrix},
 \qquad \det H_\nu=-\lambda.
\]
Thus every integer shift gives a rational isomorphism, with no exceptional value of \(\nu\).
Iteration and reflection give \eqref{eq:bessel-isomorphism-condition}.
\end{proof}

\subsection{Independence of the equations}
We put \(\nu\sim\eta\) if \(\nu-\eta\in\Z\) or \(\nu+\eta\in\Z\).
This is the equivalence relation generated by integer shifts and change of sign.
The next lemma gives the differential Galois group of the direct sum of the differential modules.
\begin{lemma}\label{lem:independent-groups}
Let \(\eta_1,\ldots,\eta_M\notin\tfrac12+\Z\) represent distinct equivalence classes for \(\sim\).
For the direct sum
\begin{equation}\label{eq:independent-family}
 \bigoplus_{\ell=0}^2\ \bigoplus_{s=1}^M B_{\omega^\ell,\eta_s}
\end{equation}
over \(\C(t)\), the differential Galois group is \(\SL_2(\C)^{3M}\).
\end{lemma}
\begin{proof}
By Proposition~\ref{prop:bessel-group}, the group \(G\) projects onto each \(\SL_2\) factor.
The identity component \(G^0\) also projects onto each factor, since the quotient \(G/G^0\) is finite and \(\SL_2\) is connected.
Its Lie algebra is therefore a subdirect product of \(3M\) copies of \(\mathfrak{sl}_2\).

A subdirect product of simple Lie algebras is a direct sum of diagonal copies on disjoint sets of isomorphic factors.
We can prove this assertion by induction: when one more simple factor is added, the kernel of its projection to the preceding product is either the full simple factor or zero; in the second case the new factor is a quotient of one simple summand of the preceding product.
If the Lie algebra of \(G^0\) is not the full direct sum, two of its projections to \(\mathfrak{sl}_2\) are related by an automorphism.
Here we use that every automorphism of \(\mathfrak{sl}_2(\C)\) is inner.
Since \(G^0\) is connected, the corresponding two standard representations are isomorphic on \(G^0\).
The resulting homomorphism matrix has entries in the fixed field of \(G^0\) inside a Picard--Vessiot field of \eqref{eq:independent-family}.
That fixed field is a finite algebraic extension of \(\C(t)\), by the differential Galois correspondence.
Proposition~\ref{prop:bessel-isomorphisms} excludes such a homomorphism between any two distinct summands of \eqref{eq:independent-family}.
Thus the Lie algebra is the full direct sum.
The ambient product \(\SL_2^{3M}\) is connected, so \(G=\SL_2^{3M}\).
\end{proof}

\section{The three coefficient subsequences}
\label{sec:rank24}
In this section, we determine the minimal differential equations of the three generating functions when the lower parameters do not satisfy \eqref{eq:integer-condition}.

We retain \(q_N\) from \eqref{eq:product-coefficients} and put
\begin{equation}\label{eq:residue-functions}
 f_r(w)=\sum_{n\geq0}q_{3n+r}w^n
 \qquad(r=0,1,2).
\end{equation}
Each \(f_r\) is entire.
Using \(\sum_{k=0}^2\omega^{k(N-r)}=3\) when \(N\equiv r\pmod3\), and zero otherwise, we obtain
\begin{equation}\label{eq:root-filter}
 f_r(z^3)=\frac{z^{-r}}3
 \sum_{k=0}^2\omega^{-kr}
 \prod_{j=0}^2{}_0F_1(;c_j;\omega^{j+k}z).
\end{equation}
The value at \(z=0\) is interpreted by its convergent series.

The three generating functions have the following common rank and irreducibility properties.
\begin{theorem}\label{thm:rank24}
Suppose that \(c_j\notin\Z_{\leq0}\cup(\tfrac12+\Z)\) for \(j=0,1,2\), and that \eqref{eq:integer-condition} fails.
For each \(r=0,1,2\), the monic minimal equation of \(f_r\) over \(\C(w)\) has order \(24\) and is irreducible.
After the ramification \(w=(t/2)^6\), its formal decomposition at infinity has the exponential factors
\begin{equation}\label{eq:rank24-exponentials}
 e^{2\zeta_6^jt}\quad(j=0,\ldots,5),\qquad 1,
 \qquad \zeta_6=\exp(\pi\sqrt{-1}/3).
\end{equation}
Each exponential factor with nonzero exponent has an associated regular singular factor of rank \(3\).
The factor with zero exponent has rank \(6\).
Equivalently, over \(\C(w)\) the part of slope \(1/6\) has rank \(18\), and the regular singular part has rank \(6\).
\end{theorem}
\begin{proof}
We put \(\nu_j=c_j-1\).
The failure of \eqref{eq:integer-condition} means that the three \(\nu_j\) do not all lie in one class for \(\sim\).
Let \(\eta_1,\ldots,\eta_M\) represent their distinct classes, where \(M=2\) or \(3\).
For each \(j\), we let \(s_j\in\{1,\ldots,M\}\) be the index determined by \(\nu_j\sim\eta_{s_j}\).
The ordered triple \((s_0,s_1,s_2)\) is not constant.

We first extend from \(\C(z)\) to \(\C(t)\), where \(z=t^2/4\).
By \eqref{eq:normalization}, each of the three products in \eqref{eq:root-filter} differs from a product of solutions of \eqref{eq:bessel-normal} by the same power
\begin{equation}\label{eq:common-power}
 t^{3/2-c_0-c_1-c_2}.
\end{equation}
For a fixed \(k\), the three normalized factors have distinct constants \(1,\omega,\omega^2\) in their equations.
Proposition~\ref{prop:bessel-isomorphisms} replaces each order \(\nu_j\) by its representative \(\eta_{s_j}\) by a rational isomorphism.
The three products, for \(k=0,1,2\), therefore correspond to the ordered triples of indices
\begin{equation}\label{eq:class-triples}
 (s_0,s_1,s_2),\qquad
 (s_2,s_0,s_1),\qquad
 (s_1,s_2,s_0)
\end{equation}
at the respective constants \(1,\omega,\omega^2\).
These choices are distinct, since a triple whose entries are not all equal has an orbit of length \(3\) under cyclic permutation.

By Lemma~\ref{lem:independent-groups}, each choice gives a tensor product of three standard representations of three distinct \(\SL_2\) factors.
Each tensor product is irreducible and has dimension \(8\).
The three choices give nonisomorphic representations: their sets of nontrivial factors in \(\SL_2^{3M}\) differ.
The equivalence between differential modules and representations \cite[Theorem~1.5.2]{Singer2009} gives three pairwise nonisomorphic irreducible differential modules of rank \(8\).
Both properties are preserved when we tensor with the module of rank \(1\) in \eqref{eq:common-power}.

We verify that the actual products generate these modules.
Multiplication defines a map from the tensor product of the three solution spaces to the common Picard--Vessiot field.
This map is equivariant for the differential Galois group.
It is nonzero, since a product of three nonzero solutions is nonzero.
Its domain is irreducible, so the map is injective.
Every nonzero product consequently has the full minimal order \(8\).
The same conclusion holds after multiplication by the power in \eqref{eq:common-power}.
The modules of the three actual products remain pairwise nonisomorphic.
These properties hold over \(\C(z)\), since they hold after extension to \(\C(t)\).

Every coefficient \(\omega^{-kr}/3\) in \eqref{eq:root-filter} is nonzero.
Multiplication by \(z^{-r}\) does not change the generated differential modules.
By Lemma~\ref{lem:module-operations}(1), the right side of \eqref{eq:root-filter} generates the direct sum of the three modules of rank \(8\) over \(\C(z)\).
Its minimal order over \(\C(z)\) is therefore \(24\).

We next descend to \(\C(w)\).
On \(\C(z)\), we use the derivation \(\partial_w=(3z^2)^{-1}\partial_z\).
The field automorphism \(z\mapsto\omega z\) commutes with this derivation and fixes \(f_r(z^3)\).
It sends the monic minimal operator of this function to another monic minimal operator for the same function.
Uniqueness makes the operator invariant.
Its coefficients therefore belong to \(\C(w)\).
No operator of smaller order exists over \(\C(w)\), since such an operator would also exist over \(\C(z)\).
Thus the minimal order over \(\C(w)\) is \(24\).

To prove irreducibility, suppose that this differential module has a nonzero proper submodule over \(\C(w)\).
After extension to \(\C(z)\), the submodule is a sum of some of the three pairwise nonisomorphic simple summands.
The deck transformation \(z\mapsto\omega z\) permutes these three summands transitively.
The extended submodule must be invariant under that permutation.
Only the zero sum and the full sum are invariant.
This contradicts the choice of a nonzero proper submodule.

Finally, we compute the formal decomposition.
For each tensor product, Lemma~\ref{lem:bessel-formal} gives the following eight coefficients of \(t\) in the exponents of its exponential factors:
\[
 \pm\sqrt{1}\ \pm\sqrt{\omega}\ \pm\sqrt{\omega^2}.
\]
Their multiset is independent of the choices of square roots.
The polynomial with these eight values as roots is
\begin{equation}\label{eq:signed-polynomial}
 X^2(X^6-64).
\end{equation}
Thus two coefficients of \(t\) in the exponents are zero, and the other six are the distinct numbers \(2\zeta_6^j\).
For the three summands, we obtain the following polynomial by taking all coefficients of \(t\) in the exponents as roots, counted with multiplicity:
\begin{equation}\label{eq:multiplicity-polynomial}
 X^6(X^6-64)^3.
\end{equation}
These coefficients of \(t\) in the exponents do not change when we multiply by the power in \eqref{eq:common-power} or by the rational factor \(z^{-r}\).
Since \(t=2w^{1/6}\), each nonzero coefficient gives slope \(1/6\) over \(\C(w)\).
Each occurs three times, and zero occurs six times.
This gives \eqref{eq:rank24-exponentials} with the asserted ranks.
\end{proof}

\begin{remark}\label{rem:rank-location}
The order \(24\) in Theorem~\ref{thm:rank24} concerns the entire generating functions in \eqref{eq:residue-functions}.
It is not an assertion about the minimal equation of the \(F_C\) restriction itself.
Over \(\C(z)\), the module of \(f_r(z^3)\) is a direct sum of three irreducible modules of rank \(8\).
The irreducibility assertion holds over \(\C(w)\).
\end{remark}

\section{Exclusion of finite generalized hypergeometric sums}
\label{sec:obstruction}
In this section, we compare the formal decomposition in Theorem~\ref{thm:rank24} with that of an ordinary generalized hypergeometric equation.
We use \(\vartheta=w\,d/dw\).

\subsection{Formal solutions of a generalized hypergeometric equation}
The following lemma describes the formal decomposition at infinity of a generalized hypergeometric differential equation.
\begin{lemma}\label{lem:hypergeometric-formal}
Let \(\alpha\ne0\), \(p,q\in\Z_{\geq0}\), and \(p\leq q+1\).
Consider
\begin{equation}\label{eq:hypergeometric-operator}
 \left[
 \vartheta\prod_{j=1}^q(\vartheta+v_j-1)
 -\alpha w\prod_{i=1}^p(\vartheta+u_i)
 \right]y=0.
\end{equation}
If \(m=q+1-p>0\), its formal decomposition at infinity has a regular singular part of rank \(p\).
Its exponential factors with nonzero exponents, after \(w=s^m\), are
\begin{equation}\label{eq:hypergeometric-exponentials}
 e^{\beta_js},\qquad
 \beta_j^m=m^m\alpha,\qquad j=0,\ldots,m-1.
\end{equation}
Each associated regular singular factor has rank \(1\).
The exponential factors with nonzero exponents have slope \(1/m\) over \(\C(w)\).
If \(m=0\), the equation is regular singular at infinity.
These assertions do not require nonintegral differences between parameters.
\end{lemma}
\begin{proof}
We use the operator \eqref{eq:hypergeometric-operator}, which is \cite[Eq.~16.8.3]{DLMF}; see also \cite[\S16.8]{NIST2010}.
Suppose first that \(m>0\), and put \(n=q+1=p+m\).
After \(w=s^m\), we have \(\vartheta=m^{-1}s\,d/ds\).
For \(y=e^{\beta s}\) with \(\beta\ne0\), the terms of highest degree in \(s\) give
\begin{equation}\label{eq:hypergeometric-characteristic}
 \beta^p(\beta^m-m^m\alpha)=0,
\end{equation}
up to a nonzero constant.
The \(m\) nonzero roots are distinct.
At each such root, substitution of
\[
 y=e^{\beta_js}s^\rho\left(1+\sum_{n\geq1}d_ns^{-n}\right)
\]
first determines \(\rho\) and then the coefficients \(d_n\).
The simple root gives one formal solution and an associated regular singular factor of rank \(1\).
For the two terms of \eqref{eq:hypergeometric-operator}, the Newton polygon has a segment of horizontal length \(m\) giving slope \(1/m\).
The remaining factor is regular singular and has rank \(p\).
Equivalently, after removing the \(m\) distinct exponential factors with nonzero exponents, the factor with zero exponent has rank \(p\).
The highest powers \(\vartheta^{p+m}\) and \(\alpha w\vartheta^p\) have nonzero coefficients for all parameter values.
Thus integer differences between parameters do not change these slopes or the simplicity of the nonzero roots.

If \(m=0\), we substitute \(w=1/x\) and multiply the equation by \(x\).
This gives an equation in \(x\,d/dx\) with holomorphic coefficients and nonzero highest coefficient at \(x=0\).
Hence infinity is regular singular.
\end{proof}

\begin{remark}\label{rem:terminating-targets}
A terminating target in \eqref{eq:pfq-definition} is a polynomial, even if its displayed numbers \(p,q\) do not satisfy \(p\leq q+1\).
A nonzero polynomial belongs to \(\C(w)\) and generates a differential module of rank \(1\) with no exponential factor with nonzero exponent.
We treat such summands by that module rather than by the full displayed hypergeometric operator.
\end{remark}

\subsection{Finite sums and composition factors}
The irreducibility in Theorem~\ref{thm:rank24} excludes a finite sum, irrespective of the number of its terms.
\begin{proposition}\label{prop:obstruction}
Under the assumptions of Theorem~\ref{thm:rank24}, none of \(f_0,f_1,f_2\) admits an expression
\begin{equation}\label{eq:product-forbidden-sum}
 f_r(w)=\sum_{\ell=1}^L R_\ell(w)\,
 {}_{p_\ell}F_{q_\ell}
 \left(\begin{matrix}\boldsymbol u_\ell\\\boldsymbol v_\ell\end{matrix};
                      \alpha_\ell w\right),
\end{equation}
where \(L\) is finite, \(R_\ell\in\C(w)\), \(\alpha_\ell\ne0\), every lower parameter of every target function avoids nonpositive integers, and every target series converges near the origin or terminates.
The equality in \eqref{eq:product-forbidden-sum} is excluded even when poles of the separate rational factors cancel in the sum.
\end{proposition}
\begin{proof}
We fix \(r\).
By Theorem~\ref{thm:rank24}, the module \(\mathcal N_{f_r}\) is irreducible of rank \(24\).
Suppose that \eqref{eq:product-forbidden-sum} holds.
By Lemma~\ref{lem:module-operations}(2), \(\mathcal N_{f_r}\) must occur as a composition factor of the differential module associated with one of the equations for the target functions.
By Remark~\ref{rem:terminating-targets}, the differential module generated by a polynomial cannot have this composition factor.

Consider a nonterminating target and put \(m=q_\ell+1-p_\ell\).
If \(m=0\), its formal decomposition has no positive slope.
If \(m>0\), its only positive slope is \(1/m\).
Since \(\mathcal N_{f_r}\) has a part of slope \(1/6\), compatibility of formal decomposition with subquotients forces \(m=6\).
In the target equation for \(m=6\), every exponential factor with nonzero exponent has an associated regular singular factor of rank \(1\), by Lemma~\ref{lem:hypergeometric-formal}.
In \(\mathcal N_{f_r}\), the six exponential factors with nonzero exponents each have an associated regular singular factor of rank \(3\).
After a common ramification, either such an exponential factor is absent from the target or its associated regular singular factor has rank at most \(1\).
The associated regular singular factor therefore cannot have rank \(3\) in a subquotient.
This contradiction excludes every individual target equation.
It also excludes the finite sum.
\end{proof}

\begin{remark}\label{rem:sum-vs-rank}
The comparison is between one irreducible composition factor and one target equation.
It is not a comparison of rank \(18\) with the total irregular rank of an arbitrarily long finite sum.
The latter rank can increase with the number of summands.
Lemma~\ref{lem:module-operations}(2) is the step that makes the exclusion apply to finite sums.
\end{remark}

The same proposition excludes a recurrence relating two consecutive coefficients by rational functions of the index.
\begin{corollary}\label{cor:no-rational-recurrence}
Under the assumptions of Theorem~\ref{thm:rank24}, fix \(r\in\{0,1,2\}\).
There do not exist nonzero polynomials \(A,B\in\C[n]\) such that
\begin{equation}\label{eq:forbidden-recurrence}
 A(n)q_{3(n+1)+r}=B(n)q_{3n+r}
\end{equation}
for all sufficiently large integers \(n\).
The coefficient subsequence is not eventually zero.
\end{corollary}
\begin{proof}
If the subsequence is eventually zero, then \(f_r\) is a polynomial.
This contradicts the minimal order in Theorem~\ref{thm:rank24}.
Suppose that \eqref{eq:forbidden-recurrence} holds.
We choose \(N\) so that \eqref{eq:forbidden-recurrence} holds for every \(n\geq N\).
We also require \(N\) to exceed every nonnegative integral zero of \(A\) or \(B\).
The recurrence implies that the tail is either identically zero or has no zero terms.
The former case has already been excluded.
In the latter case, factor \(B/A\) into linear factors over \(\C\).
Iteration expresses \(q_{3(N+n)+r}/q_{3N+r}\) as a constant to the power \(n\) times a quotient of Pochhammer symbols.
An additional upper parameter \(1\) compensates for the \(n!\) in \eqref{eq:pfq-definition}.
The lower parameters avoid nonpositive integers by the choice of \(N\).
Thus \(f_r\) is a polynomial plus \(w^N\) times one ordinary generalized hypergeometric series.
The series converges near zero, since the tail of \(f_r\) is entire.
This contradicts Proposition~\ref{prop:obstruction}.
\end{proof}

\section{Proof of the criterion}
\label{sec:main-proof}
In this section, we transfer the exclusion for the product coefficients to Lauricella's function.

\begin{proof}[Proof of Theorem~\ref{thm:main}]
Proposition~\ref{prop:construction} proves sufficiency, the \({}_9F_8\) construction, and the continuation on \(|\kappa|<1/4\).
We prove necessity by contradiction.
Suppose that \eqref{eq:finite-sum} holds and \eqref{eq:integer-condition} fails.
Since \(a,b\notin\Z_{\leq0}\), division of the coefficient of degree \(N\) by \((a)_N(b)_N\) is defined for every \(N\).
By \eqref{eq:grouping}, this operation sends the left side of \eqref{eq:finite-sum} to \eqref{eq:product-coefficients}.
By Lemma~\ref{lem:coefficient-operations}, the term with index \(\ell\) becomes
\begin{equation}\label{eq:inverse-finite-term}
 \frac{A_\ell z^{e_\ell}}{(a)_{e_\ell}(b)_{e_\ell}}
 {}_{p_\ell}F_{q_\ell+6}
 \left(\begin{matrix}\boldsymbol u_\ell\\
 \boldsymbol v_\ell,\Delta_{e_\ell}(a),\Delta_{e_\ell}(b)
 \end{matrix};\frac{\alpha_\ell z^3}{729}\right).
\end{equation}
All additional lower parameters avoid nonpositive integers.
Each transformed series is entire or is a polynomial, by Remark~\ref{rem:coefficient-convergence}.

Fix \(r\in\{0,1,2\}\).
In the coefficient subsequence of degree \(3n+r\), only terms with \(e_\ell\equiv r\pmod3\) contribute.
Dividing by \(z^r\) and putting \(w=z^3\), we obtain
\begin{equation}\label{eq:inverse-residue-sum}
 f_r(w)
 =\sum_{\substack{1\leq\ell\leq L\\e_\ell\equiv r\ (3)}}
 \frac{A_\ell\,w^{(e_\ell-r)/3}}{(a)_{e_\ell}(b)_{e_\ell}}
 {}\times
 {}_{p_\ell}F_{q_\ell+6}
 \left(\begin{matrix}\boldsymbol u_\ell\\
 \boldsymbol v_\ell,\Delta_{e_\ell}(a),\Delta_{e_\ell}(b)
 \end{matrix};\frac{\alpha_\ell w}{729}\right).
\end{equation}
The exponent \((e_\ell-r)/3\) is a nonnegative integer for every contributing term.
Thus \eqref{eq:inverse-residue-sum} is an expression excluded by Proposition~\ref{prop:obstruction}.
If the index set is empty, it asserts \(f_r=0\), which also contradicts Theorem~\ref{thm:rank24}.
This proves necessity.
\end{proof}

\begin{remark}\label{rem:rational-factors-transfer}
Proposition~\ref{prop:obstruction} allows rational factors in \(w\) on the product side.
Theorem~\ref{thm:main} allows polynomial factors in \(\kappa\) on the \(F_C\) side.
We do not transfer the assertion of Proposition~\ref{prop:obstruction} allowing rational factors by applying the inverse operation on coefficients.
Division of Taylor coefficients by \((a)_N(b)_N\) is not multiplication of functions and need not preserve a rational factor.
\end{remark}

\section{Specializations and the scope of the condition}
\label{sec:examples}
In this section, we give specializations of the criterion and of the sufficient construction.

We first take two equal lower parameters.
\begin{corollary}\label{cor:two-equal}
Suppose that \(a,b\notin\Z_{\leq0}\) and \(c,d\notin\Z_{\leq0}\cup(\tfrac12+\Z)\).
The germ
\[
 F_C^{(3)}(a,b;d,c,c;\kappa,\omega\kappa,\omega^2\kappa)
\]
has a finite expression of the form \eqref{eq:finite-sum} if and only if \(d-c\in\Z\) or \(d+c\in\Z\).
\end{corollary}
\begin{proof}
Requiring an integral sum or an integral difference defines an equivalence relation.
Thus \eqref{eq:integer-condition} for the three lower parameters \(d,c,c\) is equivalent to the stated condition.
Apply Theorem~\ref{thm:main}.
\end{proof}

The identity in Proposition~\ref{prop:reflected-product} gives an explicit formula without requiring integer differences among all lower parameters.
\begin{corollary}\label{cor:reflected-fc}
Suppose that \(c\notin\Z\cup(\tfrac12+\Z)\).
For arbitrary \(a,b\in\C\), we have the identity of germs
\begin{align}\label{eq:reflected-fc}
 &F_C^{(3)}(a,b;2-c,c,c;\kappa,\omega\kappa,\omega^2\kappa)\notag\\
 &\qquad=\sum_{r=0}^2\gamma_r(c)(a)_r(b)_r\kappa^r
 {}_8F_7\left(\begin{matrix}
 \Delta_r(a),\Delta_r(b),\boldsymbol A_r(c)\\
 \boldsymbol B_r(c)
 \end{matrix};64\kappa^3\right),
\end{align}
where \(\gamma_r(c),\boldsymbol A_r(c),\boldsymbol B_r(c)\) are the coefficients and complete parameter lists in Proposition~\ref{prop:reflected-product}.
Every lower parameter of every target function avoids nonpositive integers, and the finite sum continues the germ to \(|\kappa|<1/4\).
\end{corollary}
\begin{proof}
Apply \eqref{eq:forward-operation} to \eqref{eq:reflected-product}.
This operation does not divide by any Pochhammer symbol involving an upper parameter, so upper parameters for which the series terminates are also allowed.
Each target series converges on \(|64\kappa^3|<1\).
\end{proof}

\begin{example}\label{ex:reflected}
For \(c=2/3\), the lower parameters in \eqref{eq:reflected-fc} are \(4/3,2/3,2/3\), and
\[
 \bigl(\gamma_0(2/3),\gamma_1(2/3),\gamma_2(2/3)\bigr)
 =\left(1,-\frac34,\frac{117}{140}\right).
\]
The lower parameters \(4/3\) and \(2/3\) have nonintegral difference and integral sum.
Theorem~\ref{thm:main} also includes all shifts of these parameters by independent integers for which the original lower parameters avoid nonpositive integers.
For example, the triple of lower parameters \((2/3,4/3,7/3)\) satisfies the criterion.
In the construction we take \(c=2/3\) and write these parameters as \(c,2-c,(2-c)+1\).
\end{example}

\begin{example}\label{ex:excluded}
The triple of lower parameters \((2/3,4/3,5/4)\) does not satisfy \eqref{eq:integer-condition}.
The first two parameters have integral sum, but
\[
 \frac54-\frac23=\frac7{12},\qquad
 \frac54+\frac23=\frac{23}{12}
\]
are not integers.
For every \(a,b\notin\Z_{\leq0}\), the germ
\[
 F_C^{(3)}(a,b;2/3,4/3,5/4;\kappa,\omega\kappa,\omega^2\kappa)
\]
therefore has no expression \eqref{eq:finite-sum}.
The conclusion allows any finite number of summands and arbitrary nonnegative integers \(p_\ell,q_\ell\).
\end{example}

\begin{example}\label{ex:integers}
For the triple of lower parameters \((1,2,4)\), we can take \(C=4\) in the integer case of Proposition~\ref{prop:construction}.
The downward shifts are \((-3,-2,0)\).
At intermediate steps and in the final functions, all lower parameters avoid nonpositive integers.
No limit of singular parameter values is needed.
\end{example}

\begin{remark}\label{rem:excluded-parameters}
The assumption \(a,b\notin\Z_{\leq0}\) is needed for necessity.
For example, \(F_C^{(3)}(0,b;c_0,c_1,c_2;\kappa,\omega\kappa,\omega^2\kappa)=1\) for every triple of lower parameters avoiding nonpositive integers, including the triple in Example~\ref{ex:excluded}.
Lower parameters that are half-integers are excluded for a different reason.
The group calculation in Proposition~\ref{prop:bessel-group} changes when \(\nu\in\tfrac12+\Z\), and the independent \(\SL_2\) factors used in Theorem~\ref{thm:rank24} are then unavailable.
We do not include those cases in the classification.
\end{remark}

\begin{remark}\label{rem:finite-sum-vs-ratio}
Theorem~\ref{thm:main} classifies finite sums of functions, rather than rational ratios between consecutive coefficients in each residue class.
Corollary~\ref{cor:no-rational-recurrence} proves that such ratios cannot exist when \eqref{eq:integer-condition} fails.
We do not assert the converse for every choice of signs and integers in \eqref{eq:signed-shifts}.
\end{remark}

\section*{Acknowledgements}
The author used OpenAI Codex for computational exploration and assistance with drafting and revision.

\AtNextBibliography{\raggedright}
\printbibliography
\end{document}